\documentclass[10pt]{amsart}
\usepackage{amscd,graphicx,epsfig}
\usepackage[colorlinks=true, pdfstartview=FitV, linkcolor=blue, citecolor=blue, urlcolor=blue]{hyperref}
\usepackage{amssymb}
\usepackage{comment}
\usepackage{graphicx}

\title[Automorphism groups of Danielewski surfaces]{Automorphism groups of Danielewski surfaces}
\author{Matthias Leuenberger  and  Andriy Regeta}

\address{Universit\"at Bern, Mathematisches Institut,
 Sidlerstrasse 5,  CH-3012 Bern, Switzerland}
\email{matthias.leuenberger@bluewin.ch }

\address{Universit\"at zu K\"oln, Mathematisches Institut,
 Weyertal 86-90, D-50931 K\"oln, Germany}
\email{andriyregeta@gmail.com}

\thanks{The authors were supported by Swiss National Science Foundation (Schweizerischer National\-fonds).}

\newtheorem*{question}{Question}
\newtheorem{theorem}{Theorem}
\newtheorem{corollary}{Corollary}
\newtheorem{lemma}{Lemma}
\newtheorem{proposition}{Proposition}
\newtheorem{definition and proposition}{Definition and proposition}

\theoremstyle{definition}
\newtheorem{definition}{Definition}
\newtheorem{remark}{Remark}

\newcommand{\name}[1]{\textsc{#1\/}}

\newcommand{\Der}{\mbox{Der}}
\renewcommand{\d}{{\partial}}

\newcommand{\dx}{\frac{\partial}{\partial x}}
\newcommand{\dy}{\frac{\partial}{\partial y}}
\newcommand{\dz}{\frac{\partial}{\partial z}}

\DeclareMathOperator{\Ve}{Vec}

\DeclareMathOperator{\de}{deg}

\DeclareMathOperator{\Lie}{Lie}

\DeclareMathOperator{\Aut}{Aut}
\DeclareMathOperator{\SAut}{SAut}

\DeclareMathOperator{\GL}{GL}
\DeclareMathOperator{\SL}{SL}

\DeclareMathOperator{\Ad}{Ad}
\DeclareMathOperator{\li}{lim}
\DeclareMathOperator{\Hom}{Hom}
\DeclareMathOperator{\di}{div}
\DeclareMathOperator{\U}{U}
\DeclareMathOperator{\LND}{LND}
\DeclareMathOperator{\prr}{pr}
\def\Dp{\mathrm{D}_p}
\def\Dh{\mathrm{D}_h}
\def\Dz(z-1){\mathrm{D}_{z(z-1)}}
\def\Dq{\mathrm{D}_q}
\def\C{\mathbb{C}} 
\def\N{\mathbb{N}}

\newcommand{\lielnd}[1]{\langle \LND(#1) \rangle}
\def\O{{\mathcal O}}
\def\isorightarrow {\xrightarrow{\sim}}

\def \itt #1,#2:{\medskip\item[$\bullet$] %
     page\ \ignorespaces#1, line\ \ignorespaces#2:\ \ignorespaces}

\begin{document}

\begin{abstract}

In this note we study the automorphism group of a smooth Danie-\\lewski surface 
$\Dp= \{(x,y,z) \in \mathbb{A}^3 \mid  xy = p(z) \}   \subset \mathbb{A}^3$, where 
$p \in \mathbb{C}[z]$ is a polynomial without multiple roots and $\deg p \ge 3$.  It is known that two such generic surfaces $\Dp$ and $\Dq$ have isomorphic automorphism groups. Moreover, $\Aut(\Dp)$ is generated by algebraic subgroups and there is a natural  isomorphism $\phi \colon \Aut(\Dp) \xrightarrow{\sim}  \Aut(\Dq)$
which restricts to an isomorphism of algebraic groups $G  \xrightarrow{\sim} \phi(G)$ for any algebraic subgroup $G \subset \Aut(\Dp)$. In contrast, we prove that
$\Aut(\Dp)$ and $\Aut(\Dq)$ are isomorphic as ind-groups if and only if $\Dp \cong \Dq$ as
a variety.
 Moreover, we show that any automorphism of the ind-group $\Aut(\Dp)$ is inner. 
\end{abstract}

\maketitle


\section{Introduction and Main Results}


Our base field  is the field of complex numbers $\mathbb{C}$. For an affine algebraic variety $X$ the group of regular automorphisms 
$\Aut(X)$ has a natural structure of an ind-group (see section  \ref{Preliminaries} for the definition).  In this paper we study the following question.

\begin{question}
Let $X$ and $Y$ be affine irreducible varieties.  Assume that $\Aut(X)$ is generated by algebraic subgroups and there is an abstract isomorphism of groups $\phi: \Aut(X) \xrightarrow{\sim} \Aut(Y)$  
such that $\phi$ preserves algebraic groups, i.e. for any 
algebraic subgroup $G \subset \Aut(X)$, the image $\phi(G)$ is again an algebraic
subgroup of $\Aut(Y)$ and $\phi$ restricts to an isomorphism of algebraic 
groups $G \to \phi(G)$. Is it true that $\Aut(X)$ and $\Aut(Y)$ are isomorphic
as ind-groups?
\end{question}

It is known that for two general ind-groups, the answer  is negative. For instance,  let $W_1$ be the first Weyl algebra i.e.  the quotient of
the free associative algebra  $\mathbb{C} \langle x,y \rangle$ by the relation $xy - yx = 1$ 
 and $P_1$ be the corresponding Poisson algebra i.e.  the polynomial algebra $\mathbb{C}[x,y]$ endowed with the Poisson
 bracket  $\{ f,g \} := f_x g_y - f_y g_x$ for $f,g \in \mathbb{C}[x,y]$. 
 Notice that the automorphism group of any finitely generated algebra (not necessarily commutative) has 
 a natural structure of an ind-group (see \cite{FK15}). In particular, $\Aut(W_1)$ and $\Aut(P_1)$ have natural structures of ind-groups.
 Moreover, there is a natural  isomorphism $\phi: \Aut(W_1) \isorightarrow \Aut(P_1)$  of abstract groups (see \cite[Section 1.1]{BK05}).
 The group $\Aut(W_1)$ is isomorphic to the subgroup 
  $\SAut(\mathbb{A}^2) := \{ (F_1,F_2) \in \Aut(\mathbb{A}^2) |  \det [\frac{\d F_i}{\d x_j}]_{i,j} = 1 \} \subset \Aut(\mathbb{A}^2)$
  (see \cite{Di68}, \cite[Theorem 2]{ML84}).  It  follows from \cite{Jun42} and \cite{Kul53} (see also \cite[Theorem 2]{Kam75})  that  
  $\SAut(\mathbb{A}^2)$ is an 
  amalgamated product  of  the group $G_1 = \SL(2, \mathbb{C})  \ltimes (\mathbb{C}^+)^2$ of special affine transformations of $\mathbb{A}^2$, and the solvable group $G_2$ of polynomial transformations of the form
$$(x,y) \mapsto (\lambda x + F(y), \lambda^{-1}y), \lambda \in \mathbb{C}^*, F \in \mathbb{C}[y].$$
 In fact, it is not difficult to see  that the amalgamated product structure of $\Aut(P_1)$ and $\Aut(W_1)$  implies that the natural isomorphism $\phi$ preserves algebraic groups. On the other hand, 
  in  \cite{BK05}  \name{Belov-Kanel} and  \name{Kontsevich}   noticed that
$\Aut(W_1)$ and $\Aut(P_1)$ are not isomorphic as ind-groups.
Another interesting example comes from a natural isomorphism  $\psi: \Aut(\mathbb{C}\langle x,y \rangle) \isorightarrow \Aut(\mathbb{C}[x,y])$ of 
abstract groups (see \cite{ML70}).
Using the amalgamated product structure of the group $\Aut(\mathbb{C}[x,y])$  as above it is not difficult to see that the
 $\psi$ preserves algebraic groups. However,  \name{Furter} and \name{Kraft} \footnote{in oral communication}  showed 
 (see \cite{FK15})   that  these automorphism groups
 are not isomorphic as ind-groups.

 In both examples at least one group is not the automorphism
  group of a commutative algebra.
 The main result of this paper is a counterexample to the question posted above (see Theorem \ref{main5} and Remark \ref{remark1}).

Following \cite{AFK13}, for any affine variety $X$ we define   $\U(X) \subset \Aut(X)$ as
the subgroup which is generated by  all $\mathbb{C}^+$-actions on $X$. 
 Denote by $\mathbb{A}^2/\mu_2$   the quotient of $\mathbb{A}^2$ by the cyclic group $\mu_2 = \{ \xi \in \mathbb{C}^* | \xi^2 = 1 \}$, 
 where $\mu_2$ acts on $\mathbb{A}^2$  by  scalar multiplication,  and by $T \subset \SL_2:= \SL_2(\mathbb{C})$  
 the standard subtorus of $\SL_2$. The subgroups $\U(\mathbb{A}^2/\mu_2) \subset \Aut(\mathbb{A}^2/\mu_2)$ and
 $\U(\SL_2/T) \subset \Aut(\SL_2/T)$ are closed (see Section \ref{Section3}) and hence, they have the structure of ind-subgroups.
%
%
 In \cite[Proposition 10]{Reg15} it is shown that there is an abstract isomorphism $\phi: \U(\SL_2/T) \isorightarrow \U(\mathbb{A}^2/ \mu_2)$ which preserves algebraic groups. In contrast, we prove the following result.

\begin{theorem}\label{main}
The ind-groups $\U(\SL_2/T)$ and $\U(\mathbb{A}^2/ \mu_2)$ 
are not isomorphic.
\end{theorem}


We denote by    $\Dp = \{ (x,y,z) \in \mathbb{A}^3 | xy = p(z) \}$ the so-called  \name{Danielewski} surfaces,
where  $p$ is a polynomial in $\mathbb{C}[z]$ with  $\deg p \ge 2$ and $p$ has no multiple roots. Note that $\SL_2/T \cong \Dp$, where $p = z^2 - z$.
In the literature surfaces given by $\lbrace x^ny = p(z)\rbrace \subset \mathbb{A}^3$ are often also called \name{Danielewski} surfaces. In this text we consider only the case where $n = 1$ and the surface is smooth.


%

Denote by  $\Aut^\circ(\Dp)$ the connected component of the neutral element  of the ind-group $\Aut(\Dp)$.  
%
%
%
In order to prove Theorem \ref{main} we show that the  Lie subalgebra $\lielnd{\Dp}\subset \Ve(\Dp)$  generated by all locally nilpotent vector fields on $\Dp$ is
simple (see Proposition \ref{simple}). On the other hand the Lie subalgebra 
$\langle \LND(\mathbb{A}^2/\mu_2) \rangle \subset \Ve(\mathbb{A}^2/\mu_2)$ generated by all locally nilpotent vector fields 
is not simple. If there were an
isomorphism of ind-groups  $\U(\SL_2/T)   \isorightarrow \U(\mathbb{A}^2/ \mu_2)$, 
 then we prove that the Lie algebras $\langle \LND(\SL_2/T) \rangle$ and $\langle \LND(\mathbb{A}^2/\mu_2) \rangle$ would be isomorphic, which is not the case.

If  $\deg p > 2$,   it is proven in  \cite{ML90} that  the group $\Aut^\circ(\Dp)$ equals  $\U(\Dp) \rtimes T$, where 
$T = \{ (tx, t^{-1}y,z) | t \in \mathbb{C}^* \} \subset  \Aut^\circ(\Dp)$ is a one-dimensional subtorus and 
$\U(\Dp)$ is isomorphic to the free product $\mathbb{C}[x] \ast \mathbb{C}[y]$.
Hence, there is a natural isomorphism 
 $\psi: \Aut^\circ(\Dp) \isorightarrow \Aut^\circ(\Dq)$  of abstract groups which preserves algebraic groups
 (as follows from Proposition \ref{groupstructure}).
  On the other hand we prove the following result.

\begin{theorem}\label{main5}
The ind-groups $\Aut^\circ(\Dp)$ and $\Aut^\circ(\Dq)$ are isomorphic if and only if the varieties $\Dp$ and $\Dq$ are isomorphic.
Moreover, if $\deg p \geq 3$, then any isomorphism $\Aut^\circ(\Dp) \xrightarrow{\sim} \Aut^\circ(\Dp)$ of ind-groups is inner.
\end{theorem}

\begin{remark}\label{re}
Theorem \ref{main5} has an interesting consequence. Using this,
 one can characterize $\Dp$ by its automorphism group viewed as an ind-group i.e. to prove that for an affine normal irreducible variety $X$, isomorphism of ind-groups $\Aut(X)$ and  $\Aut(\Dp)$ implies that $X \cong \Dp$ as a variety.
\end{remark}

\begin{remark}\label{remark1}
If  $p$ is generic in the sense that no affine 
 automorphism $\alpha$ of $\mathbb{C}$ permutes the roots of $p$,
 then 
 $\Aut(\Dp)$ coincides with $\Aut^\circ(\Dp)\rtimes I$, where $I$ is the group of order two generated by the involution $(x,y,z)\mapsto (y,x,z)$.  
\end{remark}

In order to prove  Theorem \ref{main5} we show that an isomorphism of the ind-groups  $\Aut^\circ(\Dp)$ and $\Aut^\circ(\Dq)$
induces an isomorphism
 of the Lie algebras $\lielnd{\Dp}$ and $\lielnd{\Dq}$ which
 induces  an isomorphism of the \name{Danielewski} surfaces $\Dp$ and $\Dq$ (see Proposition \ref{mainprop}).  
 To prove that any automorphism of $\Aut^\circ(\Dp)$ is inner one uses Theorem \ref{main7}.

 We have the natural  identification between vector fields $\Ve(\Dp)$ and derivations $\Der( \mathcal{O}(\Dp))$ of the ring of 
 regular functions of $\Dp$.
  For
$\psi \in \Aut(\Dp)$ and $\delta \in \Ve(\Dp)$ we define 
$$\Ad(\psi) \delta := (\psi^*)^{-1}  \circ \delta \circ \psi^*$$
 where   we consider $\delta$ as a derivation $\delta:  \mathcal{O}(\Dp) \to  \mathcal{O}(\Dp)$ and  
  $\psi^*:  \mathcal{O}(\Dp) \to  \mathcal{O}(\Dp)$, $f  \mapsto f \circ \psi$, is the co-morphism of $\psi$. It is clear that the action of 
  $\Aut(\Dp)$ on 
 the Lie algebra $\Ve(\Dp)$ restricts to an action on the Lie subalgebra $\lielnd{\Dp} \subset \Ve(\Dp)$ because $\Ad(\psi)$ sends locally nilpotent
 vector fields to locally nilpotent vector fields.
 It turns out that any automorphism  of the Lie algebra $\lielnd{\Dp}$ comes from some automorphism of  $\Dp$.
  Moreover,
  from Lemma  \ref{new1}  and Proposition  \ref{prop8}  it follows that the action of $\Aut(\Dp)$ on $\Ve(\Dp)$  restricts to the action
  of $\Aut(\Dp)$ on the Lie subalgebra  $\Ve^\circ(\Dp) \subset \Ve(\Dp)$ of all vector fields of divergence zero 
  (see Section \ref{divergence0} for  definition). Vice versa, we prove the following result.

\begin{theorem}\label{main7}
If $\deg(p) \geq 3$, then we have: 

$(a)$ Let  $F$ be an automorphism of the Lie algebra $\lielnd{\Dp}$. Then
 $F$ coincides with  $\Ad(\varphi)$ for some automorphism $\varphi : \Dp \isorightarrow \Dp$.

$(b)$ Let  $F$ be an automorphism of the Lie algebra $\Ve^0{\Dp}$. Then
 $F$ coincides with  $\Ad(\varphi)$ for some automorphism $\varphi : \Dp \isorightarrow \Dp$.
\end{theorem}

This is an analogue to the following result in \cite{KR17} (see also \cite{Reg13}):
each automorphism of the Lie subalgebra 
$\langle \LND(\mathbb{A}^n) \rangle  = \{  f_1 \frac{d}{\d x_1} + ... + f_n \frac{\d}{\d x_n} \in \Ve(\mathbb{A}^n) | 
\frac{d f_1}{\d x_1} + ... + \frac{\d f_n}{\d x_n} = 0  \}  \subset \Ve(\mathbb{A}^n)$ generated by all locally nilpotent vector fields on 
$\mathbb{A}^n$  is induced from an automorphism of 
$\mathbb{A}^n$.

\textbf{Acknowledgement:} 
The authors would like to thank  \name{Hanspeter Kraft}  for fruitful discussions and support during writing this paper. 
The authors would also like to thank \name{Michel Brion}, \name{Frank Kutzschebauch}, \name{Immanuel van Santen} and
 \name{Mikhail Zaidenberg} for helpful discussions and
important remarks.

\section{Preliminaries}\label{Preliminaries}

The notion of an ind-group goes back to \name{Shafarevich} who called these objects infinite dimensional groups, see \cite{Sh66}. We refer to  \cite{Kum02} and \cite{St13}  for basic notations in this context. 

\begin{definition}
By an \emph{ind-variety}  we mean a set $V$ together with an ascending filtration $V_0 \subset V_1 \subset V_2 \subset ... \subset V$ such that the following is satisfied:

(1) $V = \bigcup_{k \in \mathbb{N}} V_k$;

(2) each $V_k$ has the structure of an algebraic variety;

(3) for all $k \in \mathbb{N}$ the subset $V_k \subset V_{k+1}$ is closed in the Zariski  topology.
\end{definition}

A \emph{morphism} between ind-varieties $V = \bigcup_k V_k$ and $W = \bigcup_l W_l$ is a map $\phi: V \to W$ such that for each $k$ there is an $l \in \mathbb{N}$ such that 
$\phi(V_k) \subset W_l$ and that the induced map $V_k \to W_l$ is a morphism of algebraic varieties. \emph{Isomorphisms} of ind-varieties are defined in the usual way.

 Filtrations $V = \bigcup_{k \in N} V_k$ and $V = \bigcup_{k \in N} V'_{k}$ are called \emph{equivalent} if for any $k$ there is an $l$ such that $V_k \subset V'_{l}$ is a closed
 subvariety as well as $V'_{k} \subset V_l$. 

An ind-variety $V = \bigcup_k V_k$ has a natural topology: $S \subset V$ is \emph{closed}, resp. \emph{open}, if $S_k := S \cap  V_k \subset V_k$ is closed, resp. open, for all k. Obviously, a closed subset $S \subset V$ has a natural structure of an ind-variety. It is called an \emph{ind-subvariety}. An ind-variety $V$ is called affine if each algebraic variety $V_k$ is affine. Later on we consider only affine ind-varieties and  for simplicity we call them just ind-varieties.

For an ind-variety $V = \bigcup_{k \in \mathbb{N}} V_k$ we can define the tangent space in $x \in V$ in the obvious way. We have  $x \in V_k$ for $k \ge k_0$, and
$T_x V_k \subset T_x V_{k+1}$ for $k \ge k_0$, and then we define
$$
  T_x V := \li_{k \ge k_0} T_x V_k,
$$
which is a vector space of countable dimension. A morphism $\phi: V \to W$ induces
linear maps $d\phi_x : T_x V \to T_{\phi(x)} W$ for any $x \in X$. Clearly, for a $\mathbb{C}$-vector space $V $ of a  countable dimension and  for any $v \in V$ we have 
$T_v V = V$ in a canonical way.


The \emph{product} of two ind-varieties is defined in the clear way. This allows us to give the following definition.

\begin{definition}
An ind-variety $G$ is called an \emph{ind-group} if the underlying set $G$ is a group such that the map $G \times G \to G$, taking $(g,h) \mapsto gh^{-1}$, is a morphism of ind-varieties.
\end{definition}

Note that any closed subgroup $H$ of $G$, i.e. $H$ is a subgroup of $G$ and is a closed subset, is again an ind-group under the  closed ind-subvariety structure on $G$.
 A closed subgroup $H$ of an ind-group $G$ is an algebraic group if and only if $H$ is an algebraic subset of $G$
 i.e. $H$ is a closed subset of some $G_i$, where $G_1 \subset G_2 \subset ...$ is a filtration of $G$.

An ind-group $G$ is called \emph{connected}  if for each $g \in G$ there is an irreducible curve $C$ and a morphism $C \to G$ whose image contains $e$ and $g$.

If $G$ is an  ind-group, then $T_e G$ has a natural structure of a Lie algebra which will be denoted by $\Lie G$. The structure is obtained by showing that each $A \in T_e G$ defines a unique left-invariant vector field $\delta_A$ on $G$, see 
\cite[Proposition 4.2.2, p. 114]{Kum02}.

The next result  can be found in \cite{St13}.
\begin{proposition}\label{ind-group}
Let $X$ be an affine variety. Then $\Aut(X)$ has a natural structure of an affine ind-group.
\end{proposition}

By $\Ve(X)$ we denote a Lie algebra of all vector fields on an affine variety $X$.
A vector field $\nu \in \Ve(X)$ is called \emph{locally nilpotent} if the corresponding   derivation $D \in \Der \O(X)$ is locally nilpotent i.e. if for any $f$ there exist $n \in \mathbb{N}$ such that $D^n(f) = 0$. Later on we always identify a vector field on an affine variety $X$ with its corresponding derivation.  By $\lielnd{X}$ we mean a Lie subalgebra of $\Ve(X)$ generated by all
locally nilpotent vector fields.

The next result can be found in \cite[Proposition 4.2.2]{Kum02}.

\begin{proposition}\label{KZ14final}
Let $\phi: G \to H$ be a homomorphism of ind-groups. Then $\phi$ induces 
a homomorphism  $d\phi_{e}  : \Lie G \to \Lie H$ of Lie algebras.
\end{proposition}

Since $\Aut(X)$ has a structure of an ind-group for any affine variety $X$, we can define a Lie algebra of $\Aut(X)$.
It is known that there is a homomorphism of Lie algebras  $\psi: \Lie \Aut(X) \to \Ve(X)$ which is injective on each 
$\Lie K \subset \Lie \Aut(X)$, where  $K \subset \Aut(X)$ is an algebraic subgroup. Hence, $\psi$ is injective on the Lie subalgebra 
$\Lie_u(\Aut(X)) := \langle \Lie K | K \subset \Aut(X) 
\text{ is an algebraic subgroup isomorphic  to } \\ \mathbb{C}^+   \rangle \subset \Lie \Aut(X)$ generated by Lie algebras of one-dimensional unipotent subgroups.
The map $\psi$ is injective on $\Lie_u(\Aut(X))$ and the image of  $\Lie_u(\Aut(X))$ under $\psi$ equals $\lielnd{X}$
because any locally nilpotent vector field $\nu \in \Ve(X)$ belongs to   $\psi(\Lie K)$ for some one-dimensional closed unipotent  subgroup 
$K$ of $\Aut(X)$. In fact, one can prove that  $\ker \psi$ is trivial.

Let $X$ and $Y$ be  affine  varieties such  that there is an isomorphism $\phi: \Aut(X) \xrightarrow{\sim} \Aut(Y)$  of ind-groups.
Then isomorphism of Lie algebras  $d\phi_{e}  : \Lie \Aut(X) \xrightarrow{\sim}  \Lie \Aut(Y)$ induces an isomorphism
of Lie subalgebras $\Lie_u(\Aut(X)) \cong \lielnd{X}$  and $\Lie_u(\Aut(Y)) \cong \lielnd{Y}$.
In the future we will always identify $\Lie_u(\Aut(X))$ with $\lielnd{X}$. 

\begin{definition}
 By $\U(X)$ we mean the subgroup of $\Aut(X)$ generated by all closed one-dimensional unipotent subgroups.
  \end{definition}

\section{Automorphisms of \name{Danielewski} surface}\label{Section3}

Let $p \in \mathbb{C}[t]$ be a polynomial of degree $d \ge 2$ with simple roots. Define the \name{Danielewski}-surface $\Dp \subset \mathbb{A}^3$ to be the zero set of the irreducible polynomial
$xy - p(z)$:
$$
\Dp = \{ (x,y,z) \in \mathbb{A}^3  | xy = p(z) \} \subset \mathbb{A}^3.
$$

The following is easy ($\dot{\mathbb{C}}  := \mathbb{C} \setminus \{ 0 \})$:

(a) $\Dp$ is smooth,

(b) the two projections $\pi_x:\Dp  \to \mathbb{C}$, $(x,y,z) \mapsto	 x$  and  $\pi_y: \Dp \to \mathbb{C}$, $(x,y,z) \mapsto y$ are both smooth,

(c) $(\Dp)_x  = \pi_x^{-1}(\dot{\mathbb{C}})   \isorightarrow  \dot{\mathbb{C}} \times \mathbb{C}$,  $(x,y,z) \mapsto (x,z)$    and similarly for  $\pi_y$, 

(d) $\pi_x^{-1}(0)$ is the disjoint union of $d$ copies of the affine line $\mathbb{C}$.
$\\$

For the rest of this section we assume $\de p >2$ unless stated otherwise. 
 For every nonzero $f \in \mathbb{C}[t]$ there is a $\mathbb{C}^+$-action $\alpha_f$ on $\dot{\mathbb{C}} \times \mathbb{C}$ given by $\alpha_f (s)(x, z) := (x, z + f(x))$, i.e. by translation with $f(x)$ in the fiber of $x \in \dot{\mathbb{C}}$. It is easy to see that this action extends to an action on $\Dp$ if and only if $f(0) = 0$. We denote the corresponding actions on $\Dp$ by 
$\alpha_{x,f}$, respectively $\alpha_{y,f}$. The explicit form is
$$
 \alpha_{x,f} (s)(x, y, z) = (x, \frac{p(z + sf (x))}{x} , z + sf (x))
$$
    and similarly for $\alpha_{y,f}$. The projection $\pi_x: \Dp \to \mathbb{C}$ is the quotient for all these actions, and the action on $\pi^{-1}(0)$ is trivial. Note that the corresponding vector fields are given by
$$
\nu_{x,f} := p'(z) \frac{f(x)}{x} \frac{\d }{\d y} + f(x) \frac{\d}{\d z} \text{ and } \nu_{y,f} := p'(z) \frac{f(y)}{y} \frac{\d}{\d x}  + f(y) \frac{\d}{\d z}. 
$$

Denote by $\U_x, \U_y \subset \Aut(\Dp)$ the image of $\alpha_x$ and $\alpha_y$. Note that there is also a faithful $\mathbb{C}^*$-action on $\Dp$ given by $t(x, y, z) := (t  x, t^{-1} · y, z)$ which normalizes $\U_x$ and $\U_y$. Denote by $T \subset \Aut(\Dp)$ the image of $\mathbb{C}^*$. The following result is due to Makar-Limanov.
\begin{proposition} \label{groupstructure}
 $(a)$ The group $\Aut(\Dp)$ is generated by $\U_x,\U_y,T$ and a finite subgroup F which normalizes $ \langle \U_x,\U_y,T \rangle$.

$(b)$ $ \langle \U_x, \U_y, T \rangle \subset \Aut(\Dp)$ is a closed connected subgroup 
of finite index, hence $\langle \U_x, \U_y, T \rangle = \Aut^\circ(\Dp)$. 


$(c)$ $\U(\Dp) = \U_x \ast \U_y$ is a free product   and every one-parameter unipotent subgroup of $\Aut(\Dp)$ is conjugate to a subgroup of $\U_x$ or $\U_y$.

$(d)$ $ \langle \U_x, \U_y, T \rangle = \U(\Dp) \rtimes T$ is a semi-direct product of $\U(\Dp)$ and torus  $T$. 
\end{proposition}
\begin{proof}
The fact that $\Aut(\Dp)$ is generated by $\U_x,\U_y,T$ and a finite subgroup is proved in \cite{ML90}. The additional claims in $(a)$, $(c)$ and $(d)$ about the structure of $\Aut(\Dp)$ are claimed in a remark of the same article and proven in \cite{KL13}. 
It is clear that the subgroup $ \langle \U_x, \U_y, T \rangle \subset \Aut(\Dp)$ is connected. Moreover, because it has a finite index, 
it is closed and $\langle \U_x, \U_y, T \rangle = \Aut^\circ(\Dp)$. Hence, $(b)$ follows.
\end{proof}

\begin{proposition}\label{Blanc}
$\Aut(\SL_2/T) = \U(\SL_2/T) \rtimes \mu_2$, where $\mu_2$ denotes a cyclic group of order $2$. In particular, $\Aut^\circ(\SL_2/T) = \U(\SL_2/T)$ is an ind-group.
\end{proposition}
\begin{proof}
It is clear that $\SL_2/T \cong \mathrm{D}_{z(z-1)}$.  Note that for any two polynomials $p,q \in \mathbb{C}[z]$ of degree $2$ without multiple roots,  we have  $\Dp \cong \Dq$. 
It follows from \cite[Theorem 6]{Lam05} that $\Aut(\mathrm{D}_{z(z-1)})$ is generated by $\mathbb{C}^+$-actions and cyclic subgroup  $\mu_2$ of order $2$ which permute roots $\{ a, b \}$ of $p$,  i.e 
$\Aut(\mathrm{D}_{z(z-1)}) = \langle  \U(\mathrm{D}_{z(z-1)}), \mu_2 \rangle $.
Because $\U(\mathrm{D}_{z(z-1)})$ is normal  subgroup of $\Aut(\mathrm{D}_{z(z-1)})$,  we have  
$\Aut(\mathrm{D}_{z(z-1)}) = \U(\mathrm{D}_{z(z-1)}) \rtimes \mu_2$. 

From the decomposition $\Aut(\SL_2/T) = \U(\SL_2/T) \sqcup g \U(\SL_2/T)$, where $g \in \mu_2$ is a non-neutral element, it follows that
$\Aut^\circ(\SL_2/T) \subset \U(\SL_2/T)$. Now, the second statement follows from the fact that $\U(\SL_2/T)$ is generated by $\mathbb{C}^+$-actions.
\end{proof}

Note that $\U(\mathbb{A}^2/\mu_2) \subset \Aut(\mathbb{A}^2/\mu_2)$ is a closed subgroup (see \cite[page 9]{Reg15}). Hence, $\U(\mathbb{A}^2/\mu_2)$ has the natural structure of an ind-group.

\section{Proof of the Main Results}
We denote by  $\lielnd{\mathbb{A}^2/\mu_2} $ the Lie subalgebra of $\Ve(\mathbb{A}^2/\mu_2)$ generated by all locally nilpotent vector fields on $\mathbb{A}^2/\mu_2$.

\begin{proof}[Proof of  Theorem \ref{main}]
 Assume there is an isomorphism $\phi: \U(\SL_2/T) \xrightarrow{\sim} \U(\mathbb{A}^2/\mu_2)$ of ind-groups.
Then it induces an isomorphism $d \phi_{e} : \Lie \U(\SL_2/T) \xrightarrow{\sim} \Lie \U(\mathbb{A}^2/\mu_2)$  of  Lie algebras, and because $\phi$ maps each closed unipotent subgroup $U \cong \mathbb{C}^+$ to $\phi(U) \cong \mathbb{C}^+$, $d \phi_{e}$ induces an isomorphism of Lie algebras $\lielnd{ \SL_2/T}$ and $ \lielnd{\mathbb{A}^2/\mu_2}$.

By Proposition \ref{simple}, $\lielnd{\SL_2/T}$ is simple.
On the other hand,  we claim that $\lielnd{\mathbb{A}^2/\mu_2}$ is not simple. Indeed, since $\mathbb{A}^2/\mu_2$ has an isolated singular point $s$,  each vector field, which comes from  an  algebraic group action, vanishes at this singular point. In particular, each locally nilpotent vector field vanishes at $s$.    Because $\lielnd{\mathbb{A}^2/\mu_2}$ is generated by locally nilpotent vector fields, each  $\nu \in  \lielnd{\mathbb{A}^2/\mu_2}$ vanishes at $s$.
Let $I \subset  \lielnd{\mathbb{A}^2/\mu_2}$ be a Lie subalgebra generated by those vector fields which vanish at $s$ with multiplicity $k >1$. It is clear that
$I \neq  \lielnd{\mathbb{A}^2/\mu_2}$ because $x \dy \in \lielnd{\mathbb{A}^2/\mu_2} \setminus I$.  Moreover, it is easy to see that $[\nu,\mu] \in I$
for any $\nu \in I$ and any $\mu \in \lielnd{\mathbb{A}^2/\mu_2}$  which shows that $I$ is an ideal.  The claim follows.
%
%
%
%
%
\end{proof}

\begin{proof}[Proof of Theorem \ref{main5}]
Assume there is an isomorphism $\phi: \Aut^\circ(\Dp) \xrightarrow{\sim} \Aut^\circ(\Dq)$ of ind-groups.
Then it induces an isomorphism $d \phi_{e} : \Lie \Aut^\circ(\Dp) \xrightarrow{\sim} \Lie \Aut^\circ(\Dq)$  of  Lie algebras, and because $\phi$ maps each closed unipotent subgroup $H \cong \mathbb{C}^+$ to $\phi(H) \cong \mathbb{C}^+$, $d \phi_{e}$ induces an isomorphism of Lie algebras $\lielnd{\Dp}$ and $\lielnd{\Dq}$. Now  from Proposition \ref{mainprop}(c) it follows that $\Dp \cong \Dq$.

Let now $\Dq = \Dp$ and  $\phi: \Aut^\circ(\Dp) \xrightarrow{\sim} \Aut^\circ(\Dp)$ be an automorphism of an ind-group
$ \Aut^\circ(\Dp)$. 
The map $d \phi_{e}$ induces an automorphism of the Lie algebra $\lielnd{\Dp}$. 
By Proposition \ref{mainprop}, $d \phi_{e} \nu = \Ad(F)  \nu$, where $\nu$ is a locally nilpotent vector field
and $F: \Dp \xrightarrow{\sim} \Dp$ is an automorphism of the variety.
This implies that $\phi(h) = F^{-1} \circ h \circ F$ for any $h \in H$, where $H \subset \Aut^\circ(\Dp)$ is a closed subgroup isomorphic  
to $\mathbb{C}^+$. Hence, the restriction of  $\phi$ to  $\U(\Dp)$ is an inner automorphism. Now the claim follows from Proposition \ref{remarkC*}.
%
%
\end{proof}

Let $\theta: X \xrightarrow{\sim}  Y$ be an isomorphism of affine varieties.
For $\delta \in \Ve(X)$ we define
$$\Ad(\theta) \delta = (\theta^*)^{-1} \circ \delta \circ \theta^*,$$
where we consider $\delta$ as a derivation $\delta:  \mathcal{O}(X) \to  \mathcal{O}(X)$ and  
$\theta^* :  \mathcal{O}(Y) \to  \mathcal{O}(X)$, 
$f  \mapsto f \circ \theta$, is the co-morphism of $\theta$. 
 It is clear that $\Ad(\theta): \Ve(X) \xrightarrow{\sim}  \Ve(Y)$  is an isomorphism of Lie algebras. 
 Moreover, because  $\Ad(\theta)$ sends locally nilpotent vector fields to locally nilpotent, $\Ad(\theta)$ induces an isomorphism of Lie algebras $\langle \LND(X) \rangle \subset \Ve(X)$ and 
$\langle \LND(Y) \rangle \subset \Ve(Y)$.

The next crucial result we will prove in  Section \ref{last}. By $\Ve^0{\Dp}$ we denote the Lie subalgebra of $\Ve{\Dp}$ of
 volume preserving vector fields (see Section \ref{Duality.}).

\begin{proposition}\label{mainprop}
$(a)$ Let  $F: \lielnd{\Dp} \ \isorightarrow \ \lielnd{\Dq}$ be an isomorphism of Lie algebras and let $\deg(q) \geq 3$. 
Then $F$ is equal to $\Ad(\varphi_F)$, where $\varphi_F : \Dp \isorightarrow \Dq$ is an isomorphism of varieties.

$(b)$ Let  $F: \Ve^0{\Dp} \ \isorightarrow \ \Ve^0{\Dq}$ be an isomorphism of Lie algebras and let $\deg(q) \geq 3$. 
Then $F$ is equal to $\Ad(\varphi_F)$, where $\varphi_F : \Dp \isorightarrow \Dq$ is an isomorphism of varieties.

$(c)$ If Lie algebras $\lielnd{\Dp}$ and $\lielnd{\Dq}$ are isomorphic then the varieties $\Dp$ and $\Dq$ are isomorphic.
\end{proposition}

\begin{proof}[Proof of Theorem \ref{main7}]
The proof follows from Proposition \ref{mainprop}.
\end{proof}

\section{Module of differentials and vector fields}

Since $\Dp$ is smooth, the differentials $\Omega(\Dp)$ and the vector fields $\Ve(\Dp) \isorightarrow \Hom(\Omega (\Dp), \mathcal{O}(\Dp))$ are locally free
$\mathcal{O}(\Dp)$-modules, and then, projective. More precisely, we have the following description.

\begin{proposition}\label{proposition2}
	$(a)$ The module $\Omega(\Dp)$ of differentials is projective of rank $2$ and  is  generated by $dx, dy, dz$, with the unique relation $ydx + xdy - p'(z)dz = 0$.

$(b)$ The module $\Omega^2(\Dp) := 􏰬\bigwedge^2 \Omega(\Dp)$ is free of rank one  and  is   generated by 
$$
\omega_p:= \frac{1}{x} dx\wedge dz= -\frac{1}{y}dy \wedge dz= \frac{1}{p'(z)} dx\wedge dy.
$$
\end{proposition}
\begin{proof}

(a) From above it is clear that $\Omega(\Dp)$ is the projective module of rank $2 = \dim(\Dp)$.
It is easy to see that $\Omega(\Dp) = (\mathcal{O}(\Dp)dx \oplus \mathcal{O}(\Dp)dy \oplus \mathcal{O}(\Dp)dz)/  (ydx   + xdy - p'(z)dz)$, where $ydx   + xdy - p'(z)dz = d(xy - p(z))$.
In fact, the surface $\Dp$ is covered by the special open sets $D_x,D_y,D_{p'(z)}$
and $\Omega(\Dp)$  is free module of rank  two over these open sets, generated by $(dx,dz)$, by $(dy,dz)$, and by $(dx, dy)$, respectively.

(b)  
The three expressions are well-defined in the special open sets $D_x, D_y, D_{p'(z)}$, respectively, and the relation $ydx + xdy - p'(z)dz = 0$ implies that they coincide on the intersections. Thus $\omega_p$ is a nowhere vanishing section of $\Omega^2(\Dp)$ and therefore, $\Omega^2(\Dp)$ is free of rank 1 (see also \cite[Section 3]{KK10} for details). 
\end{proof}

\begin{remark}
 In fact, for any normal hypersurface $X \subset \mathbb{C}^n$, $\Omega^{n-1}(X) := \bigwedge^{n-1} \Omega(X)$ is a free module of rank one.
\end{remark}

\begin{remark}
 Note that there is no $\delta \in \Ve(\Dp)$ such that $\delta: \mathcal{O}(\Dp) \to \mathcal{O}(\Dp)$ is surjective  because $\Omega(\Dp)$ is not free. Note also that $\omega_p$ is unique up to a constant   because 
$\mathcal{O}(\Dp)^* = \mathbb{C}^*$.
\end{remark}

It is well-known that every vector field $\delta$ on $\Dp \subset \mathbb{C}^3$ extends to a vector field $\tilde{\delta}$ on $\mathbb{C}^3$. It follows that $\delta$ can be written in the form
$$
\delta=a \dx +b\dy +c \dz,
$$
where $a,b,c \in \mathcal{O}(\Dp)$  such that $ay + bx - cp'(z) = 0$ in $\mathcal{O}(\Dp)$. In fact, considering $\delta$ as a $\mathcal{O}(\Dp)$-linear map $\Omega(\Dp) \to \mathcal{O}(\Dp)$, we have $a = \delta(dx)$, $b = \delta(dy)$  and $c = \delta(dz)$.  This presentation of $\delta$ is unique.

\begin{remark} \label{notation-vector-fields}
In fact, 
the vector fields $\Ve(\Dp)$ form a  module over $\mathcal{O}(\Dp)$ of rank $2$, generated by
$$
\nu_z:=x \dx -y \dy, \; \; \nu_x:=p'(z)\dy +x\dz, \; \; \nu_y:=p'(z)\dx +y\dz
$$
with the unique relation $x\nu_y - y\nu_x = p'(z)\nu_z$. 
\end{remark}

The next result is clear.
\begin{proposition}\label{proposition4}
 The sequence 
$$
0 \to \mathbb{C} \to \mathcal{O}(\Dp) \xrightarrow{d} d \Omega(\Dp) \xrightarrow{d} d \Omega^2(\Dp) \to 0
$$
is exact.
\end{proposition}

\section{Volume form and divergence.}\label{divergence0}

 For any $\theta \in \Ve(\Dp)$ we have the contraction 
$$
i_{\theta} : \Omega^{k+1} \to \Omega^k, \; \; i_{\theta}(\eta)(\theta_1,...,\theta_k) := \eta(\theta,\theta_1,...,\theta_k).
$$
In particular, for $\eta \in \Omega(\Dp)$, we have $i_{\theta}(\eta) = \eta(\theta) \in \mathcal{O}(\Dp)$, and so $i_{\theta}(df) = \theta_f$. 

The vector field $\theta \in \Ve(\Dp)$ acts on the differential forms $\Omega(\Dp)$ by the Lie
derivative $L_{\theta} := d \circ i_{\theta} + i_{\theta} \circ d$, extending the action on $\mathcal{O}(\Dp)$. One finds (see for details \cite[Section 3]{KK10})
$$
L_{\theta}(f) = \theta f, \;  L_{\theta}(df) = d(\theta f) \text{ and } L_{\theta}(h \cdot \mu) = \theta h \cdot \mu+h\cdot L_{\theta} \mu \text{ for } f,h \in \mathcal{O}(\Dp), \mu \in \Omega(\Dp).
$$
Using the volume form $\omega_p$ (see Proposition \ref{proposition2}),  this allows to define the divergence $\di(\theta)$ of a vector field $\theta$:
$$
L_{\theta}\omega_p = d(i_{\theta} \omega_p) = \di(\theta) \cdot \omega_p.
$$ 

The following Lemma is well-known:
\begin{lemma}
Let $\theta = a\dx +b\dy +c\dz \in \Ve(\Dp)$. Then $\di(\delta)=a_x+b_y+c_z$.
\end{lemma}

\begin{proof}
 We have $i_{\theta} \omega_p = \frac{1}{x}(\theta(x)dz-\theta(z))dx= \frac{1}{x}(adz-cdx)$, hence

\begin{align*} 
\di(\theta) \cdot \omega_p = d(i_{\theta} \omega_p) =& d(\frac{1}{x}(adz-cdx)) \\
=&	\frac{1}{x^2} ((xda-adx) \wedge dz-(xdc \wedge dx).  
\end{align*}
Now we use the following equalities: $da \wedge dz = a_x \cdot dx \wedge dz + a_y \cdot dy \wedge dz$, $dx \wedge dc = c_y \cdot dx\wedge dy + c_x  \cdot dx\wedge dz$, $dy \wedge dz = -y \cdot \omega_p$, and $dx \wedge dy = p'(z) \cdot \omega_p$ (see above) to get
$$
\di(\theta)=-\frac{a}{x} + a_x - \frac{y}{x} a_y + \frac{p'(z)}{x}c_y + c_z.
$$ 
Since $ya+xb-p'(z)c=0$ we have $a+ya_y +xb_y - p'(z)c_y  = 0$, hence 
$$
-\frac{a}{x} - \frac{y}{x} a_y + \frac{p'(z)}{x} c_y = b_y,
$$
and the claim follows.
\end{proof}

There is another important formula which relates the Lie structure of $\Ve(\Dp)$ with the Lie derivative (see also \cite[Lemma 3.2]{KL13}).

\begin{lemma} For $\theta_1,\theta_2 \in \Ve(\Dp)$ and $\mu \in \Omega(\Dp)$ we have
$$
i_{[\theta_1,\theta_2]} \mu = L_{\theta_1} (i_{\theta_2} \mu) - i_{\theta_2} (L_{\theta_1} \mu).
$$
\end{lemma}

\section{Duality.}\label{Duality.}

The volume form $\omega_p \in \Omega^2(\Dp)$ induces the usual duality between vector 
fields and differential forms: the $\mathcal{O}(\Dp)$-isomorphism $\Ve(\Dp) \isorightarrow \Omega(\Dp)$ is given by 
$\theta \mapsto i_{\theta} \omega_p$. In particular, for every $f \in \mathcal{O}(\Dp)$ we get a vector field 
$\nu_f \in \Ve(\Dp)$ defined by $i_{\nu_f} \omega_p = df$.

Denote by $\Ve^0(\Dp) \subset \Ve(\Dp)$ the subspace of volume preserving vector fields, i.e. $\Ve^0(\Dp) := \{ \theta \in \Ve(\Dp) | \di \theta = 0 \}$.


\begin{proposition}\label{prop8}
 The map $f \mapsto \nu_f$ induces a $\mathbb{C}$-linear isomorphism 
$$
\mathcal{O}(\Dp)/\mathbb{C} \isorightarrow  \Ve^0(\Dp).
$$
\end{proposition}

 \begin{proof} Since $d(i_{\theta}\omega_p‰) = \di(\theta) \cdot \omega_p$, we have the following commutative diagram:

$$\begin{CD}
0 @>{}>>  \mathbb{C} @>{d}>> \mathcal{O}(\Dp) @>{d}>> \Omega(\Dp) @>{d}>> \Omega^2(\Dp)  @>{}>> 0 \\
 @.   @.   @AA{=} A    @A{\simeq}A{\theta \mapsto i_{\theta} \omega_p}A@A{\simeq}A{h \mapsto h \cdot \omega_p}A \\
@ .  @.   \mathcal{O}(\Dp) @>{\nu}>> \Ve(\Dp)  @>{\di}>>  \mathcal{O}(\Dp)
\end{CD}
$$
Now the claim follows   because the first row is exact (see Proposition \ref{proposition4}).
\end{proof}

The following result can be found in \cite[Theorem 3.26]{KL13}.

\begin{proposition}\label{correspondents}
Any vector field $\nu \in \Ve^0(\Dp)$ on the Danielewski surface $\Dp$ is a Lie combination of locally nilpotent vector fields if and only if its corresponding function with $i_{\nu} \omega_p = df$ is of the form (modulo constant)
\begin{equation}\label{functions}    f(x,y,z)= \sum_{i=1, j=0}^{k} a_{ij}x^iz^j + \sum_{i=1, j=0}^{l} b_{ij}y^iz^j +(pr)'(z)  \end{equation}
for a polynomial $r \in \mathbb{C}[z]$. If $\deg p \ge 3$, then  any $f \in \mathcal{O}(\Dp)/\mathbb{C}$ bijectively corresponds to some $\nu_f \in  \Ve^{0}(\Dp) := \lielnd{\Dp} \oplus \bigoplus_{i=0}^{\deg p -3} \mathbb{C} z^i(x\dx - y\dy)$.
\end{proposition}

Some corresponding functions are given as follows
(see \cite[Lemma 3.1]{KL13}): Let $h$ be a polynomial in one variable and let $\nu_x,\nu_y, \nu_z$ be the vector fields from Remark \ref{notation-vector-fields}. Then
\begin{equation}\label{corresp} f_{h'(x)\nu_x} = -h(x), \quad f_{h'(y)\nu_y}= h(y), \quad f_{h'(z)\nu_z}= h(z). \end{equation}

We also recall the useful relation that describes the corresponding function of a Lie bracket of two vector fields $\nu,\mu\in \Ve^\circ \Dp$ (see  \cite[Lemma 3.2]{KL13}):
\begin{equation}\label{bracket}
 f_{[\nu,\mu]} = \nu(f_\mu),
\end{equation}
where $\nu(f_\mu)$ is $\nu$ applied as a derivation to the function $f_\mu$. The function $f_{[\nu,\mu]}$ may also be calculated by the following formula (see 
\cite[formula after Lemma 3.2]{KL13}):
\begin{flalign}
\label{bracketformula}f_{[\nu,\mu]} = \lbrace f_\nu,f_\mu \rbrace := \ 
&p'(z)\big((f_\nu)_y(f_\mu)_x-(f_\nu)_x(f_\mu)_y\big) + \\& 
\nonumber x\big((f_\nu)_z(f_\mu)_x - (f_\nu)_x(f_\mu)_z\big) -y\big((f_\nu)_z(f_\mu)_y-(f_\nu)_y(f_\mu)_z\big),
\end{flalign}
where the subindex denotes the partial derivative to the respective variable.

Let $I\subset \lielnd{\Dp}$ be a non-trivial ideal and let $\tilde I$ be the set of functions corresponding to this ideal by the correspondence in (\ref{functions}).
Since $I$ is an
ideal,  we have, using (\ref{bracket}), that
\begin{equation}\label{ideal}
 \left\lbrace\nu \in I \left(\Leftrightarrow f_\nu\in\tilde I\right) \text{ and } \ \mu \in\lielnd{\Dp} \right\rbrace\implies \nu(f_\mu), \ \mu(f_\nu) \in \tilde I.
\end{equation}

Our next goal is to prove the following result.

\begin{proposition}\label{simple}
The Lie algebra $ \lielnd{\Dp}$ is  simple.
\end{proposition}

We prove Proposition \ref{simple} in  several steps and start with the following Lemma.

\begin{lemma}\label{Step1}
Let $f$ be a regular function on $\Dp$.  Then $f$ can be written uniquely as $f(x,y,z) = \sum_{i=l}^{k} a_i(z)x^i$ for some $k,l \in \mathbb{Z}$. 
\end{lemma}
\begin{proof}
Let us take the form of $f$ as in (\ref{functions}) and replace $y$ by
$p(z)/x$.  The claim follows.   
\end{proof}

Choose $l,k \in \mathbb{Z}$ such that $a_l,a_k \neq 0$  and  denote by $\deg(f) = (l,k)$ the pair of min- and max-degree in $x$.

\begin{lemma}\label{Step2}
 Let  $f \in \mathcal{O}(\Dp)$.  Then $\nu_x(f)$ and $\nu_y(f)$  are never non-zero constants.
\end{lemma}
\begin{proof}
Any regular function $f$ on $\Dp$ can be written in the form  $\sum_{i=l}^{k} a_i(z)x^i$ by Lemma \ref{Step1}. Then $\nu_x(f) = \sum_{i=l}^{k}  a'_i(z)x^{i+1}$, in particular, $\nu_x(f)$ is constant only if $a_{-1}$ is  linear,  which is not the case since  $a_{-1}$ is divisible by $p$.  The  case of $\nu_y(f)$ is analogous.
\end{proof}

\begin{lemma}\label{Step4}
Let  $\deg f = (l,k)$ and  $l,k \ge 1$. Then  $\deg \nu_y(f)=(l-1,k-1)$.
\end{lemma}
\begin{proof}
Let $ f = \sum_{i=l}^{k} a_i(z)x^i$.
Then $\nu_y(f) = \sum_{i=l}^{k} (i p'(z) a_i(z) + p(z) a'_i(z))x^{i-1}$.  
If $a_i(z) \neq 0$, then $i p'(z) a_i(z) + p(z) a'_i(z) \neq 0$ and the claim follows. 
\end{proof}

\begin{lemma}\label{Step3}
Let $\deg f = (l,k)$, where $k>l\ge 0$, then $\deg \left(p''(z)\nu_z(f)\right) = (\tilde{l},k)$, where  $\tilde{l} = l$  if $l \ge 1$ and  $\tilde{l} > l$ if  $l = 0$.  
\end{lemma}
\begin{proof}
Let $f = \sum_{i=l}^{k} a_i(z)x^i$.  Then the claim follows  from the equality  $\nu_z(f) = \sum_{i=\tilde{l}}^{k} i  a_i(z)x^i$.
\end{proof}

\begin{lemma}\label{ExistsH}
Let $\tilde{I}\neq 0$. Then there exists some $h\in\tilde{I}$ such that $h\in\C[z]\setminus\C$ is a non-constant polynomial in $z$.
\end{lemma}
\begin{proof}
Take a non-zero $f \in \tilde{I}$. Since $\nu_x$ is locally nilpotent, there is $k \in \mathbb{N}$ such that $\nu_x(\nu_x^{k}(f)) = 0$. Thus we have a function $g = \nu_x^{k}(f) \in \mathbb{C}[x] \setminus \mathbb{C}$, which means $\deg g = (l,k)$ for $k,l \ge 1$. From (\ref{ideal}) 
it follows that $g\in\tilde{I}$. By applying  Lemma \ref{Step4} and Lemma \ref{Step3} step by step,  we will get that here is a non-constant $h \in \tilde{I}$ such that 
$\deg h = (0,0)$, which proves the claim.
\end{proof}

\begin{lemma}\label{AllStartingN}
Let $\tilde{I}\neq 0$. Then there is some $m\in\N$ such that $r(z)x^{n+1} \in \tilde{I}$ for any $r  \in \mathbb{C}[z]$ and $n \ge m$.
\end{lemma}
\begin{proof}
By Lemma \ref{ExistsH} we have a non-constant $h(z)\in\tilde{I}$.
By (\ref{ideal}),  we get that $\nu_x(h) = h'(z)x  \in  \tilde{I}$ and $(p(z)s(z))''\nu_z(h'(z)x) = (p(z)s(z))''h'(z)x \in \tilde{I}$ for any $s \in \C[z]$. 
Let $m = \deg p''h'$ and $n \ge  m$.  Then  applying  (\ref{ideal}) $m-1$ times for $\mu = \nu_x$  we get
$$
\nu_x^{m-1}  ((p(z)s(z))'' h'(z)x) = ((ps)'' h' )^{(m-1)} (z)x^m \in \tilde{I}.
$$
Now apply (\ref{ideal}) once more for $\mu = x^{n-m} \nu_x$ and  get
$$
x^{n-m} \nu_x(((ps)''h')^{(m-1)}(z)x^m ) = ((ps)''h')^{(m)}(z) x^{n+1} \in \tilde{I},
̃$$
and thus varying $s(z)$ we get $r(z)x^{n+1} \in \tilde{I}$  for any $r \in \mathbb{C}[z]$.
\end{proof}

\begin{lemma}\label{AllStartingOne}
Let $\tilde{I}\neq 0$. Then $r(z)x^{n} \in \tilde{I}$ for all $r  \in \mathbb{C}[z]$ and $n \ge 1$.
\end{lemma}
\begin{proof}
We prove this lemma by induction. 
By Lemma \ref{AllStartingN} we know that  $r(z)x^{n} \in \tilde{I}$ for any $r  \in \mathbb{C}[z]$ and $n \ge m$. 
Using (\ref{ideal}) we have the following:
$$
\nu_y(z^{j}x^{m}) = m p'(z)z^{j}x^{m-1} + jp(z)z^{j-1} x^{m-1} \in \tilde{I}
$$
for all $j \in \mathbb{N} \cup \{ 0 \}$. On the other hand, since $x^m \in \tilde{I}$,  from (\ref{corresp}) it follows that  
$x^{m-1}\nu_x \in \tilde{I}$.  Hence,  by (\ref{bracketformula}) and (\ref{ideal}), 
$$
x^{m-1} \nu_x(yz^j) = p'(z)z^j x^{m-1} + jp(z)z^{j-1}x^{m-1} \in \tilde{I}
$$
for all $j \in \mathbb{N} \cup \{ 0 \}$. By taking suitable linear combinations of the above expressions we see that
$x^{m-1} \cdot (p'(z)) \subset \tilde{I}$ and $x^{m-1} \cdot  (p(z)) \subset \tilde{I}$, where $(p'(z))$ and $(p(z))$ denote the ideals
in $\mathbb{C}[z]$ generated by $p'(z)$ and $p(z)$ respectively. Since $p$ has only simple roots,  the ideal $(p'(z),p(z))$ generated by both $p'(z)$ and $p(z)$ is equal to 
$\mathbb{C}[z]$ and thus $x^{m-1} \cdot (p'(z),p(z)) =x^{m-1} \cdot \mathbb{C}[z] \subset \tilde{I}$. Therefore, $r(z)x^{n} \in \tilde{I}$ for any $r  \in \mathbb{C}[z]$ and $n \ge m-1$. The claim follows.
\end{proof}

\begin{proof}[Proof of Proposition \ref{simple}]
	Let $I$ be a nontrivial ideal of $\lielnd{\Dp}$ and let $\tilde{I}$ be the corresponding ideal on the level of functions.  By Lemma \ref{AllStartingOne} we have that $r(z)x^{n} \in \tilde{I}$ for any $r  \in \mathbb{C}[z]$ and $n\in\N$.
	Analogously,  interchanging $x$ and $y$ in Lemmas  \ref{AllStartingN} and \ref{AllStartingOne}, we get that $r(z)y^{n} \in \tilde{I}$ for any $r(z) \in \mathbb{C}[z]$ and $n \in \mathbb{N}$. 
	Since $r(z)x \in \tilde{I}$ for any $r(z)$, from (\ref{bracketformula}) it follows that $\nu_y(r(z)x) = (p(z)r(z))' \in \tilde{I}$ for any $r(z)$. Thus,
	 $\tilde{I}$ contains all functions that correspond to vector fields in $\lielnd{\Dp}$ or, equivalently,  $I = \lielnd{\Dp}$,
	which concludes the proof.
\end{proof}

\section{Proof of Proposition \ref{mainprop}.}\label{last}

Let $p,q$ be two polynomials with simple zeroes and let $\deg p \ge 2$ and  $\deg q \geq 3$ unless stated otherwise. Furthermore, let 
$\deg p \le \deg q$. Let us denote by $\Dp=\lbrace xy=p(z) \rbrace$ and by
$\Dq=\lbrace uv=q(w)\rbrace$. On $\Dq$ we introduce similar vector fields as on $\Dp$: 
\[
 \nu_u  =  q'(w)\frac{\partial}{\partial v} + u \frac{\partial}{\partial w}, \quad
\nu_v  =  q'(w)\frac{\partial}{\partial u} + v \frac{\partial}{\partial w}, \quad
\nu_w  =  u\frac{\partial}{\partial u} - v \frac{\partial}{\partial v}.
\]

We start with a simple Lemma:
\begin{lemma} \label{lndtolnd}
	Let $F$ be homomorphism from either $\lielnd{\Dp}$ or $\Ve^0(\Dp)$ to $\Ve^0(\Dp)$ and let $\nu\in\Ve^0(\Dp)$.
 The kernel of $\nu$ is mapped by $G$ onto the kernel of $F(\nu)$. In particular, if $\nu$ is locally nilpotent, then $F(\nu)$ is locally nilpotent.
\end{lemma}
\begin{proof}
	Let $G$ be the induced homomorphism on the level of functions.
	Let $f\in\ker\nu$ then by (\ref{bracket})
	$$F(\nu)(G(f)) = \lbrace f_{F(\nu)},G(f) \rbrace = \lbrace G(f_\nu),G(f) \rbrace = G(\lbrace\ f_\nu,f\rbrace) = G(\nu(f))=G(0)=0.$$
	If $\nu$ is  locally nilpotent then for every $f\in \mathcal{O}(\Dp)$ we have $\nu^i(f)=0$ for some $i\geq1$. Thus, by a similar calculation as above using (\ref{bracket}) 
	we get $F(\nu)^i(G(f))=0$. Therefore, for every $f\in  \mathcal{O}(\Dq)$ we have $F(\nu)^i(f)=0$ for some $i\geq 0$ meaning that $F(\nu)$ is locally nilpotent.
\end{proof}

For the rest of the paper we assume that $$F:\lielnd{\Dp} \quad  \isorightarrow \quad     \lielnd{\Dq}$$ 
is an isomorphism of Lie algebras and  $G$   
 is the induced 
isomorphism on the level of functions using the correspondence (\ref{functions}). The following three lemmas will be needed in the future. Recall that here $\deg(f) = (l,k)$ denotes the pair of min- and max-degree in $u$, where $f \in  \mathcal{O}(\Dq)$. The following Lemma is clear by direct calculations.

\begin{lemma} \label{lemBiDegree}
Let $f,g\in \mathcal{O}(\Dq)$ with $\deg(f) = (\alpha^-,\alpha^+)$ and $\deg(g) = (\beta^-,\beta^+)$. Then $\deg(fg) = (\alpha^- + \beta^-, \alpha^+ + \beta^+)$. Therefore, if $f$ is divisible by $g$ we get $\deg(f/g) = (\alpha^- - \beta^-, \alpha^+ - \beta^+)$. Let $p(t)$ be a polynomial of degree $n$. Then $\deg(p(f)) =(n\alpha^-,n\alpha^+)$.
\end{lemma}

\begin{lemma}\label{lemdeg}
 Let $n=\deg(q)\geq 3$, and let $\varphi = (\varphi_u,\varphi_v,\varphi_w) \in \U(\Dq)=\U_u  \ast \U_v$. Let us denote $\deg(\varphi_u) = 
(\alpha^-,\alpha^+)$ and $\deg(\varphi_v)=(\beta^-,\beta^+)$. If $\alpha^-\geq -1$ then $\varphi_u=u$ and if $\beta^-\geq -1$ then $\varphi_v=v$ or $\varphi_v = 
u^{-1}q(w + ua(u))$ for some $a\in\C[u]$.
\end{lemma}

\begin{proof} Let $\varphi=\varphi_N\circ\ldots\circ\varphi_1$ be the reduced composition  of automorphisms, where
$\varphi_i \in \U_u$ or $\varphi_i \in \U_v$.
 This means that 
$\varphi_i$ 
is given either by (where $a_i$ is a nonzero polynomial)
\begin{align*}&\varphi_i(u,v,w) = (u,\ u^{-1}q(w+ua_i(u)),\ w + ua_i(u)) \quad \text{or} \\ &\varphi_i(u,v,w)= (v^{-1}q(w+va_i(v)),\ v,\ w+va_i(v))\end{align*}
depending whether it is in $\U_u$ or in $\U_v$. 

First we consider the case when $\varphi_1\in \U_u$, meaning that $\varphi_i\in \U_u$ when $i$ is odd and 
$\varphi_i\in \U_v$ when $i$ is even. Let us take a look how the degree of the components evolves when composing with the next
automorphism. 
Let 
$$ (u_i,v_i,w_i) = \varphi_i\circ \ldots\circ\varphi_1(u,v,w)$$ and $$(\deg u_i, \deg v_i, \deg w_i)= 
((\alpha_i^-,\alpha_i^+),(\beta_i^-,\beta_i^+),(\delta_i^-,\delta_i^+))$$
We will see that the complexity of the components (measured by the degree) is growing. Looking at the first two compositions we get:
\begin{eqnarray*}
 (u_0,v_0,w_0)&=&(u,v,w)\\
 (u_1,v_1,w_1)&=&(u,u^{-1}q(w+ua_1(u)),w+ua_1(u))\\
 (u_2,v_2,w_2)&=&(v_1^{-1}q(w_1+v_1a_2(v_1)),v_1,w_1+v_1a_2(v_1))
\end{eqnarray*}
We will use Lemma \ref{lemBiDegree} to calculate the degrees of the component. Lemma \ref{lemBiDegree} will be used for all calculations in this proof. Let $n_i=\deg a_i+1$, then we get:
\begin{eqnarray*}
 \Big((\alpha_0^-,\alpha_0^+),(\beta_0^-,\beta_0^+),(\delta_0^-,\delta_0^+)\Big) &=& \Big((1,1),(-1,-1),(0,0)\Big)\\
\Big((\alpha_1^-,\alpha_1^+),(\beta_1^-,\beta_1^+),(\delta_1^-,\delta_1^+)\Big) &=& \Big((1,1),(-1,nn_1-1),(0,n_1)\Big)
\end{eqnarray*}
The computation of $\Big((\alpha_2^-,\alpha_2^+),(\beta_2^-,\beta_2^+),(\delta_2^-,\delta_2^+)\Big)$ is more tedious: Since $\deg w_1 = (0,n_1)$ and $\deg (v_1a_2(v_1)) = (-n_2,n_2(nn_1-1))$ we see that $\deg(w_1 + v_1a_2(v_1))= \deg (v_1a_2(v_1))$. Therefore, $\deg(q(w_1 + v_1a_2(v_1))) = (-nn_2,nn_2(nn_1-1))$. Together with $\deg v_1 = (-1, nn_1 - 1)$ this yields $(\alpha_2^-,\alpha_2^+) = (-nn_2+1,  (nn_2-1)(nn_1-1))$. Similar considerations show

$$ \Big((\alpha_2^-,\alpha_2^+),(\beta_2^-,\beta_2^+),(\delta_2^-,\delta_2^+)\Big) \quad \quad =$$ 
$$ \Big((-nn_2+1,(nn_2-1)(nn_1-1)),(-1,nn_1-1),(-n_2,n_2(nn_1-1))\Big).$$

These calculations show the claim for the case $N=1,2,3$ with $\varphi_1\in \U_u$. For larger $N$ it is enough to show that $\alpha_i^-<-1$ and $\beta_i^-<-1$ 
for all $4\leq i\leq N$. We show it by induction. We claim that 
$\alpha_{2k}^-<\delta_{2k}^- \le 0$, $\alpha_{2k}^+>\delta_{2k}^+ \ge 0$, $\beta_{2k-1}^-<\delta_{2k-1}^-\le 0$ and $\beta_{2k-1}^+>\delta_{2k-1}^+ \ge 0$ for $k\geq 1$. For $k=1$ 
these inequalities follow from the above. For the step $(k-1)\rightsquigarrow k$ we need the following computations:
\begin{eqnarray*}
&&(u_{2k-1},v_{2k-1},w_{2k-1}) \quad = \\&& (u_{2k-2},u_{2k-2}^{-1}q(w_{2k-2}+u_{2k-2}a_{2k-1}(u_{2k-2})),w_{2k-2}+u_{2k-2}a_{2k-1}(u_{2k-2}))\\&&
 (u_{2k},v_{2k},w_{2k}) \quad = \\ &&(v_{2k-1}^{-1}q(w_{2k-1}+v_{2k-1}a_{2k}(v_{2k-1})),v_{2k-1},w_{2k-1}+v_{2k-1}a_{2k}(v_{2k-1}))
\end{eqnarray*}

Let us, for example, make the calculation of $\deg v_{2k-1} = (\beta_{2k-1}^-,\beta_{2k-1}^+)$. We have $\deg w_{2k-2} = (\delta_{2k-2}^-,\delta_{2k-2}^+)$ and $$\deg (u_{2k-2}a_{2k-1}(u_{2k-2})) = (n_{2k-1}\alpha_{2k-2}^-,n_{2k-1}\alpha_{2k-2}^+).$$ Since $\alpha_{2k-2}^- < \delta_{2k-2}^-$ and $\alpha_{2k-2}^+>\delta_{2k-2}^+$ by induction, we get $$\deg(w_{2k-2} + u_{2k-2}a_{2k-1}(u_{2k-2})) = (n_{2k-1}\alpha_{2k-2}^-,n_{2k-1}\alpha_{2k-2}^+).$$
Together with $$\deg u_{2k-2} = (\alpha_{2k-2}^-,\alpha_{2k-2}^+)$$ we conclude that
$$ (\beta_{2k-1}^-, \beta_{2k-1}^+) = (nn_{2k-1}\alpha_{2k-2}^- -\alpha_{2k-2}^-, nn_{2k-1}\alpha_{2k-2}^+-\alpha_{2k-2}^+).$$

Similar calculations show:
\begin{eqnarray*}
  && \Big((\alpha_{2k-1}^-,\alpha_{2k-1}^+),(\beta_{2k-1}^-,\beta_{2k-1}^+),(\delta_{2k-1}^-,\delta_{2k-1}^+)\Big) \quad =\\ 
&&\Big((\alpha_{2k-2}^-,\alpha_{2k-2}^+),((nn_{2k-1}-1)\alpha_{2k-2}^-, (nn_{2k-1}-1)\alpha_{2k-2}^+),\\ && \hspace{4cm}(n_{2k-1}\alpha_{2k-2}^-,n_{2k-1}\alpha_{2k-2}^+)\Big)\\
  && \Big((\alpha_{2k}^-,\alpha_{2k}^+),(\beta_{2k}^-,\beta_{2k}^+),(\delta_{2k}^-,\delta_{2k}^+)\Big) \quad =\\ 
&&\Big(((nn_{2k}-1)\beta_{2k-1}^-, (nn_{2k}-1)\beta_{2k-1}^+),(\beta_{2k-1}^-,\beta_{2k-1}^+),\\ && \hspace{4cm}(n_{2k}\beta_{2k-1}^-,n_{2k}\beta_{2k-1}^+)\Big)
\end{eqnarray*}

Since $\alpha_{2k-2}^-<0$ by induction, we see that $$\beta_{2k-1}^- = (nn_{2k-1}-1)\alpha_{2k-2}^- < n_{2k-1}\alpha_{2k-2}^- $$
is negative because  $n_{2k-1}\alpha_{2k-2}^- = \delta_{2k-1}^- < 0.$   The inequalities $\alpha_{2k}^-<\delta_{2k}^- \le 0$, $\alpha_{2k}^+>\delta_{2k}^+ \ge 0$ and $\beta_{2k-1}^+>\delta_{2k-1}^+ \ge 0$ follow in a similar way. From the same calculations we  deduce that 
$\alpha_i^-\leq \alpha_{i-1}^-$, and $\beta_i^-\leq 
\beta_{i-1}^-$. Together with the fact that $\alpha_2^-<-1$ and $\beta_3^-<-1$ this leads as desired to $\alpha_i^-<-1$ and $\beta_i^-<-1$ 
for all $4\leq i\leq N$. Therefore, the claim of the Lemma follows in the case 
$\varphi_1\in \U_u$.

For the case $\varphi_1\in \U_v$ a similar calculation shows that $\varphi_u = u$ whenever $\alpha^-\geq -1$ and that $\varphi_v = v$ whenever $\beta^-\geq -1$.
\end{proof}

\begin{lemma}\label{new1}
	Let $\varphi: \Dp \to \Dq$ be an isomorphism and $f\in \mathcal{O}(\Dp)$. Then $\Ad(\varphi)\nu_f = \nu_{\tilde{f}}$, for some 
	$\tilde{f} \in  \mathcal{O}(\Dq)$.  Moreover, 
	 $f = \frac{\varphi^*\tilde{f}}{k(\varphi)}$, where $k(\varphi) \in \C$ is given by $\varphi^*\omega_q = k(\varphi)\omega_p$. 
\end{lemma}
\begin{proof}
	In order to find the corresponding function $f \in \mathcal{O}(\Dp)$ of the vector field $\Ad(\varphi)^{-1}\nu_{\tilde{f}}$ we need to find 
	$f$ such that $d f = i_{\Ad(\varphi)^{-1}\nu_{\tilde{f}}}\omega_p$.
	The calculation
		$$ d\left(\frac{\varphi^*\tilde{f}}{k(\varphi)}\right) = \varphi^*\left( \frac{d\tilde{f}}{k(\varphi)}\right) =\varphi^*\left( \frac{i_{\nu_{\tilde{f}}} \omega_q}{k(\varphi)}\right)  = \varphi^* \left(i_{\nu_{\tilde{f}}}\left( \frac{\omega_q}{k(\varphi)}\right)\right) $$
	$$ = \varphi^*i_{\nu_{\tilde{f}}}\left( (\varphi^{-1})^*\omega_p\right) = i_{(\Ad(\varphi))^{-1}\nu_{\tilde{f}}}\omega_p
	= i_{\nu_f} \omega_p. $$	
	shows that $f = \frac{\varphi^*\tilde{f}}{k(\varphi)}$ is the desired formula.
\end{proof}


In the next two lemmas we will show that there is an automorphism $\varphi : \Dq \isorightarrow \Dq$ such that up to composition with 
$\Ad(\phi)$, $F$ has a certain form. 

\begin{lemma}\label{lndstandard}
 Up to composition with some automorphism $\Ad(\varphi)$ of $\lielnd{\Dq}$ we have $G(x)=f(u)$ for some polynomial $f \in \mathbb{C}[u]$. 
\end{lemma}

\begin{proof}
 Since $\nu_x$ is locally nilpotent, $F(\nu_x)$ is also locally nilpotent by Lemma \ref{lndtolnd}. From Proposition \ref{groupstructure}, 
 it follows that $F(\nu_x)$ is conjugate to  $r(u)\nu_u$ by  $\Ad(\varphi)$ for some automorphism $\varphi: \Dq \to \Dq$,  where 
 $r \in \mathbb{C}[u]$.  Hence, by (\ref{corresp}) we get $G(x)=f(u)$.
\end{proof}

Note that equality $G(x)=f(u)$ implies that $F(\nu_x) = f'(u)\nu_u$.


\begin{lemma}\label{xy}
 Up to composition with an induced automorphism $\Ad(\varphi)$ of $\lielnd{\Dq}$ we have $G(x)=u$ and $G(y)=cv$ for some $c\in\C^*$.
\end{lemma}
\begin{proof}
By the previous Lemma we can assume that $G(x)=f(u)$ for some polynomial $f \in \mathbb{C}[u]$.
 Let $m=\deg p +1$. Thus, we have $\nu_x^m(y)=0$, and hence $F(\nu_x)^m(G(y))= f'(u)^m \nu_u^m(G(y)) =0$. From $\nu_u^m(G(y))=0$ and $\deg p\leq \deg q$ we 
 conclude that $G(y)=\sum_{i=0}^{m-1} a_i(u)w^i + \lambda(u) v$ for some $a_i\in\C[u]$ and $\lambda  \in\C[u]$. Hence, we have $\deg G(y) = (\alpha^-,\alpha^+)$ with 
$\alpha^-\geq -1$. Since $\nu_y$ is locally nilpotent, by Lemma \ref{lndtolnd}  $F(\nu_y)$ is also locally nilpotent. 
 Let us 
choose $\varphi\in\U(\Dq)$ such that $({\Ad(\varphi))^{-1}}F(\nu_y)$ is either in $\U_u$ or in $\U_v$ (see \cite[Theorem 2.15]{KL13}). This means by Lemma \ref{new1} that $(\Ad(\varphi))^{-1}G(y)$ is either a polynomial in $u$ or in $v$. Therefore,  $G(y)=g(\varphi_u)$ 
or $G(y)=g(\varphi_v)$ for $\varphi=(\varphi_u,\varphi_v,\varphi_w)$ and some polynomial $g$. Since $\alpha^-\geq -1$ and since $g$ is a polynomial either the min-degree in $u$ of $\varphi_u$ or $\varphi_v$ is greater or equal than $-1$, which is exactly the assumption of Lemma \ref{lemdeg}. Thus, either $G(y)=g(u)$, $G(y)=g(v)$ or $G(y)=g(u^{-1}q(w+ua(u)))$ for some $a\in\C[u]$. The first case can be excluded since $[\nu_x,\nu_y]\neq 0$ implies that
$[f'(u)\nu_u,F(\nu_y)] \neq 0$. In the latter two cases we directly see that $\deg g =1$, because the min-degree in $u$ of both $v$ and $u^{-1}q(w+ua(u))$ is equal to $-1$. Since the correspondence (\ref{corresp}) is only 
up to constants we have $g(t)=c_2t$ for some $c_2 \in \C$. In the second case we have $G(y) = c_2 v$. In the third case we define the automorphism 
$$\varphi:  (u,v,w) \mapsto (u,u^{-1}q(w+ua(u)),w+ua(u))$$
and after composition with $(\Ad(\varphi))^{-1}$ we get $G(x)=f(u)$ and $G(y)=c_2v$. 

Since $\nu_y^m(x)=0$ we have that $\nu_v^m(f(u))=0$. Thus, $\deg f = 1$ and we have $f(u)=c_1u$ for some $c_1\in\C$. After composing
$G$ with the induced map of
$$\psi(u,v,w)=(c_1u,c_1^{-1}v,w)$$
the claim follows.
\end{proof}

Since now on we assume $F$ and respectively $G$ to be as in Lemma \ref{xy}.

\begin{lemma}\label{z}
 We have $G(p'(z))=cq'(w)$. Moreover, there is a basis $\{ h_i(w) \mid i \ge 0 \}$ of $\C[w]$ such that $G((p(z)z^i)')=c(q(w)h_i(w))'$.
\end{lemma}
\begin{proof}
 The first statement follows from a direct calculation using (\ref{corresp}),(\ref{bracket}) and Lemma \ref{xy}:
$$ G(p'(z))=G(f_{[\nu_x,\nu_y]}) =  \{  f_{F(\nu_x)}, f_{F(\nu_y)}  \}  = F(\nu_x)(G(y)) = \nu_u(cv) = cq'(w). $$
The second statement holds since the kernel of $p''(z)\nu_z$ is mapped by $G$ onto the kernel of $F(p''(z)\nu_z)=cq''(w)\nu_w$ 
by Lemma \ref{lndtolnd}.
\end{proof}

\begin{lemma}\label{xzyz}
 We have that $G(xz^i) = u(aw+b)^i$, $G(yz^i)=cv(aw+b)^i$ and $G((p(z)z^i)')= c(q(w)(aw+b)^i)'$ for some $a,c \in \mathbb{C}^*$, $b \in \mathbb{C}$.
\end{lemma}
\begin{proof}
 We have $\nu_y(xz^i) = (p(z)z^i)'$, and thus $c\nu_v(G(xz^i))= c(q(w)h_i(w))'$ by Lemma \ref{xy} and  Lemma \ref{z}. This can happen only if $G(xz^i) = 
uh_i(w) + A(v)$ for some $A\in\C[v]$. Similarly, we have $\nu_x(yz^i) = (p(z)z^i)'$, and thus $\nu_u(G(yz^j)) = c(q(w)h_i(w))'$, which only happens if $G(yz^j) = 
cvh_j(w) + B(u)$ for some $B\in\C[u]$.

 From (\ref{bracketformula}) we get
$$ \lbrace xz^i,yz^j\rbrace = - (p(z)z^{i+j})'.$$
Thus by the above Lemma we have
$$ \lbrace uh_i(w) + A(v),cvh_j(w) + B(u) \rbrace = -c(q(w)h_{i+j}(w))'.$$
On the other hand, by (\ref{bracketformula}) we have 
\begin{eqnarray*} 
&\lbrace uh_i(w) + A(v),cvh_j(w) + B(u) \rbrace =&\\& - c(q(w)h_i(w)h_j(w))' + q'(w)A'(v)B'(u) + u^2h_i'(w)B'(u) + cv^2h_j'(w)A'(v).&
\end{eqnarray*}
These equations can only hold if $A(v)=B(u) = 0$ and if $h_i(w)h_j(w)=h_{i+j}$ for all $i,j\geq 1$. From the latter we can conclude that $h_i(w) = (h_1(w))^i$ 
for some $h_1(w)$. Because  $\{ h_i(w) | i\geq 0 \}$  form a basis of $\C[w]$ we know that $h_1(w)=aw+b$ is of degree 1. Thus the claim follows.
\end{proof}

\begin{lemma}\label{16}
 We have $G(x^i)=a^{i-1}u^i$ and $G(y^i)=a^{i-1}(cv)^i$.
\end{lemma}
\begin{proof}
 We know that functions in $x$ are mapped by $G$ onto functions in $u$ since the $\ker\nu_x$ is mapped onto $\ker\nu_u$ by Lemma \ref{lndtolnd}. Let 
$f_i(u)=G(x^i)$.
 From $ ix^{i+1} =  \{  x^i,xz \} $ we get $$if_{i+1}(u)=   \{ G(x^i), G(xz) \} = f_i'(u)\nu_u(G(xz))=f_i'(u)\nu_u(u(aw+b))=af_i'(u)u^2$$
using (\ref{corresp}) and Lemma \ref{xzyz}.
Similarly we get for $g_i(v)=G(y^i)$
 $$ig_{i+1}(v)=g_i'(v)\nu_v(G(yz))=g_i'(v)\nu_v(cv(aw+b))=acg_i'(v)v^2. $$
Since $f_1(u)=u$ and $g_1(v)=cv$ by Lemma \ref{xy} the claim follows.
\end{proof}

\begin{lemma}\label{pq}
 We have $ca^2q(w)=p(aw+b)$.
\end{lemma}
\begin{proof}
 We have $x\nu_x(y^2)= (p^2(z))'$ which yields $au\nu_u(ac^2v^2) = c(q(w)p(aw+b))'$ by the Lemma \ref{xzyz} and 
 Lemma \ref{16} . On the other hand, we have 
$au\nu_u(ac^2v^2)=a^2c^2(q^2(w))'$. Hence, the claim follows.
\end{proof}

\begin{lemma}\label{new2}
	Let $F_1,F_2: \lielnd{\Dp} \ \isorightarrow \ \lielnd{\Dq}$ be two isomorphisms of Lie algebras such that $F_1(x^{i}\nu_x) = F_2(x^{i}\nu_x)$ and $F_1(y^{i}\nu_y) = F_2(y^{i}\nu_y)$ for all $i \ge 0$. Then $F_1 = F_2$.
\end{lemma}
\begin{proof}
	The claim follows directly from the fact that the vector fields $x^{i}\nu_x$ and $y^{i}\nu_y$ generate $\lielnd{\Dp}$ (see \cite{KL13}).
\end{proof}

\begin{lemma} \label{identityonsubalgebra}
	Let $F_1, F_2: \Ve^0{\Dp} \ \isorightarrow \ \Ve^0{\Dq}$ be two isomorphisms of Lie algebras such that restrictions of $F_1$ and 
	$F_2$  to $\lielnd{\Dp}$ coincide. Then $F_1 = F_2$.
\end{lemma}
\begin{proof}
Let $G_1$ and $G_2$ be the isomorphism on the level of functions which correspond to $F_1$ and $F_2$ respectively.  
  Let also $g\in\C[z] $ be a polynomial.  From  Proposition \ref{correspondents} and (\ref{bracketformula})
  it follows that $ \lbrace f, g(z) \rbrace = g'(z) \lbrace f, z \rbrace$ has form (\ref{functions}) for any $f \in  \mathcal{O}(\Dp)$ such that $\nu_f \in \lielnd{\Dp}$.
   Since  restrictions of $F_1$ and 
	$F_2$  to $\lielnd{\Dp}$ coincide,
   $G_1(f) = G_2(f)$ and $G_1(\lbrace f, g(z) \rbrace) = G_2(\lbrace f, g(z) \rbrace)$
   for any $f$
  such that $\nu_f \in \lielnd{\Dp}$. Hence,  
  $\lbrace G_1(f), G_1(g(z)) - G_2(g(z)) \rbrace) = 0$ which implies that $\lbrace h, G_1(g(z)) - G_2(g(z)) \rbrace) = 0$
  for any $h \in \mathcal{O}(\Dq)$ such that $\nu_h \in  \lielnd{\Dq}$. 
 Using (\ref{bracketformula}) it is easy to see that the equality  $\lbrace u, G_1(g(z)) - G_2(g(z)) \rbrace)
 = \nu_u(G_1(g(z)) - G_2(g(z)) = 0$ implies that $G_1(g(z)) - G_2(g(z)) \in \mathbb{C}[u]$.
  Analogously, because $\lbrace v, G_1(g(z)) - G_2(g(z)) \rbrace) = 0$ we have 
  $G_1(g(z)) - G_2(g(z)) \in \mathbb{C}[v]$. Therefore, $G_1(g(z)) - G_2(g(z)) \in \mathbb{C}$. The claim follows
  from Proposition \ref{correspondents}.
%
 %
%
%
\end{proof}

\begin{proof}[Proof of Proposition \ref{mainprop}]
	$(a)$: 
	 By Lemma \ref{16}, we know that up to composition with an  automorphism $\Ad(\varphi)$ of $\lielnd{\Dq}$, where 
	 $\varphi: \Dq \to \Dq$ is some automorphism,  the following holds: there are $a,c\in\mathbb{C}^*$ and $b\in\mathbb{C}$ such that $G(x^{i}) = a^{i-1}u^{i}$ and $G(y^{i}) = a^{i-1}(cv)^{i}$ for all $i$. Let us define an isomorphism 
	 $\psi: \mathbb{A}^3 \isorightarrow \mathbb{A}^3$ in the following way: $\psi: (x,y,z) \mapsto (\frac{1}{a}u,\frac{1}{ac}v,\frac{w-b}{a})$. By Lemma \ref{pq}, $ca^2q(w)=p(aw+b)$. Thus, $\psi$ restricts to an isomorphism $\Dp \isorightarrow \Dq$ and 
	 then $\Ad(\psi): \lielnd{\Dp} \isorightarrow \lielnd{\Dq}$ is the isomorphism of the Lie algebras. Let $G_\psi$ be the corresponding map on the level of functions. By Lemma \ref{new1}, we have $G_\psi(x^{i}) = (au)^{i}/k(\psi)$ and $ G_\psi(y^{i}) = (acv)^{i}/k(\psi)$. Since $\psi^*\omega_q = a\omega_p$ and thus $k(\psi) = a$, we have $G(x^i)=G_\psi(x^i)$ and $G(y^i) = G_\psi(y^i)$ for all $i\geq 1$. Therefore, the conditions of Lemma \ref{new2} are satisfied, thus $F = \Ad(\psi)$.
	
	The proof of $(c)$ follows directly from $(a)$ when either $\deg p$ or $\deg q$ is greater or equal than $3$. Now assume that $\deg p=\deg q = 2$. Since any $\Dh$ is isomorphic to $\Dz(z-1)$ (see \cite{DP09}), where $\deg h = 2$, $(c)$ follows.

	$(b)$ Let  $F: \Ve^0{\Dp} \ \isorightarrow \ \Ve^0{\Dq}$ be an isomorphism of Lie algebras and let $\deg q \geq 3$. The map $F$ restricts to an isomorphism $\tilde{F}:  \lielnd{\Dp} \ \isorightarrow \ \lielnd{\Dq}$, because $F$ and $F^{-1}$ send locally nilpotent vector fields to locally nilpotent vector fields by Lemma \ref{lndtolnd}. From $(a)$ it follows that there is an isomorphism $\varphi: \Dp \isorightarrow \Dq$ such that 
	$\tilde{F} = \Ad(\varphi)$. 
	Hence, from Lemma \ref{identityonsubalgebra}  it follows that   $F=\Ad(\varphi)$ on $\Ve^0(\Dp)$.
\end{proof}

The next statement is used in the proof of Theorem \ref{main5}.

\begin{proposition}\label{remarkC*}
Let $\phi: \Aut^\circ(\Dp) \isorightarrow \Aut^\circ(\Dp)$ be an automorphism of an  ind-group  such that
restriction of $\phi$ to $\U(\Dp)$ is the identity map. Then $\phi$ is the identity.
\end{proposition}
\begin{proof}
 Since the restriction of $\phi$ to $\U(\Dp)$ is the identity map,  $\phi(T)$ acts on $\U(\Dp)$ by conjugations in the same way as 
 $T$ acts on $\U(\Dp)$.
Hence,  $\Ad(t)$ and $\Ad(\phi(t))$
act on $\lielnd{\Dp}$ in the same way, where $t \in T$. 
This is equivalent to the statement that $\Ad(\psi^{-1} \circ \phi(t))$ acts identically on $\lielnd{\Dp}$. We claim that
then $t^{-1} \circ \phi(\psi)$ is a trivial automorphism of $\Dp$. Indeed, because $\Ad(\psi^{-1} \circ \phi(\psi))$ acts identically on $\lielnd{\Dp}$ it follows that $(t^{-1} \circ \phi(t))^*(x) = x + c$ for some $c \in \mathbb{C}$.  Hence, 
$(\psi^{-1} \circ \phi(\psi))^*(x^2) = (x+c)^2  = x^2 + 2cx$. Because  $\Ad(t^{-1} \circ \phi(t)) (x\nu_{x})$ has to be equal to
$x\nu_{x}$ it follows that $c = 0$.
Hence, 
$(t^{-1} \circ \phi(t))^*(x) = x$. Analogously, $(t^{-1} \circ \phi(t))^*(y) = y$ and 
$(t^{-1} \circ \phi(t))^*(z) = z$. 
Therefore, $\phi(t) = t$ for any $t \in  T$. The claim follows from Proposition \ref{groupstructure}.
\end{proof}


\par

\vskip1cm

\end{document}